\documentclass[11pt]{article}

\usepackage{graphicx, verbatim}
\usepackage{amssymb}
\usepackage{alltt}
\usepackage{amsmath, amssymb, amsfonts, amscd, xspace, pifont, amsthm}
\usepackage{mathrsfs}
\usepackage{algorithm}
\usepackage{enumitem}   
\newtheorem{theorem}{Theorem}
\newtheorem{lemma}{Lemma}
\newtheorem{proposition}{Proposition}
\newtheorem{corollary}{Corollary}

\RequirePackage[numbers]{natbib}

\newcommand{\V}[1]{\ensuremath{\boldsymbol{#1}}\xspace}
\def\threeImages#1#2#3#4#5#6#7#8#9 
{
	\centerline{\hfill\makebox[#2]{#3}\hfill\makebox[#5]{#6}\hfill\makebox[#8]{#9}\hfill}
	\centerline{\hfill
		\includegraphics[width=#2]{#1}
		\hfill
		\includegraphics[width=#5]{#4}
		\hfill
		\includegraphics[width=#8]{#7}
		\hfill}
}

\def\twoImages#1#2#3#4#5#6 
{
	\centerline{\hfill\makebox[#2]{#3}\hfill\makebox[#5]{#6}\hfill}
	\centerline{\hfill
		\includegraphics[width=#2]{#1}
		\hfill
		\includegraphics[width=#5]{#4}
		\hfill}
}
\theoremstyle{remark}

\begin{document}
	
	\title {An Optimal Uniform Concentration Inequality for Discrete Entropies on Finite Alphabets in the High-dimensional Setting}
	
	\author{
		Yunpeng Zhao \thanks{School of Mathematical and Natural Sciences,
			Arizona State University, AZ, 85306. \texttt{Email:} yunpeng.zhao@asu.edu.} %%%Acknowlegements
	} %%%Acknowledgements	
\maketitle
\begin{abstract}
	We prove an exponential decay concentration inequality to bound the tail probability of the difference between the log-likelihood of discrete random variables on a finite alphabet and the negative entropy. The concentration bound we derive holds uniformly over all parameter values. The new result improves the convergence rate in an earlier result of Zhao (2020), from $(K^2\log K)/n=o(1)$ to $ (\log K)^2/n=o(1)$, where $n$ is the sample size and $K$ is the size of the alphabet. We further prove that the rate $(\log K)^2/n=o(1)$ is optimal. The results are extended to misspecified log-likelihoods for grouped random variables. We give applications of the new result in information theory. 
\end{abstract}

{\it Keywords:}  Concentration inequality; log-likelihood; entropy; typical set; source coding theorem; non-convex optimization

\section{Main result}
As a powerful toolset in probability theory \cite{boucheron2013concentration,ledoux2001concentration}, concentration inequalities have wide applications in statistics \cite{wainwright_2019,vershynin2018high}, information theory \cite{raginsky2012concentration}, and algorithm analysis \cite{dubhashi2009concentration}. The information entropy, or just entropy, is one of the central concepts in information theory \cite{cover2006elements}.  The goal of the present paper is to prove a concentration inequality with the optimal rate to bound the tail probability of the difference between the log-likelihood of discrete random variables and the negative entropy  when the number of possible values of the variable grows. In this section, we first state the main result of the paper. We will explain the motivation of this work and review related work in Section \ref{sec:motivation}. 

Let $X$ be a discrete random variable with values in a finite alphabet $\mathcal{X}=\{x_1,...,x_K\}$ and probability mass function $\{p_k=\mathbb{P}(X=x_k) \}_{k=1,...,K}$. %Let $P(X)$ be a random variable with $P(X)=p_k$ if $X=x_k$ for $k=1,...,K$. T
The entropy of $X$ is defined as 
$
-\sum_{k=1}^K p_k \log p_k.
$
Note that the definition of entropy \footnote{Throughout the paper, ``$\log$'' denotes the natural logarithm and ``$\log_2$'' denotes the logarithm base 2.} does not depend on the possible values of the variable $\{x_1,...,x_K\}$ but only depends the probabilities of taking each value $\{p_1,...,p_K\}$. In fact, $\mathcal{X}$ is not necessary a set of real numbers.  The set can contain symbols such as letters, and is named as alphabet in the information theory literature \cite{cover2006elements}. One can therefore equivalently define entropy  on a categorical variable.  Let $\V{z}=(z_{1},...,z_{K})$ be a dummy coding of a categorical variable with $K$ categories, in which one and only one entry is 1 and the others are 0. Let $\V{p}=(p_1,...,p_K)$ with $p_{k}=\mathbb{P}(z_{k}=1)=\mathbb{P}(X=x_k), k=1,...,K$. The log-likelihood of $\V{z}$ is $L(\V{z})=\sum_{k=1}^K z_{k} \log p_k$ and the negative entropy of $\V{z}$ is $\mathbb{E}[L(\V{z})]=\sum_{k=1}^K p_k \log p_k$.

Given a sequence of independent and identically distributed  random variables $\V{z}_1,...,\V{z}_n$, a natural question is to derive a concentration bound for the difference between the mean of  log-likelihoods of $\V{z}_1,...,\V{z}_n$ and its expectation, i.e., the negative entropy.  We consider a slightly more general setting, in which the variables are assumed to be independent but not necessarily identical. Specifically, let $\V{z}_i=(z_{i1},...,z_{iK})$ follow a categorical distribution with parameters $\V{p}_i$ where $p_{ik}=\mathbb{P}(z_{ik}=1),i=1,...,n, k=1,...,K$. We assume   $\V{z}_1,...,\V{z}_n$ are  independent but $\V{p}_i$ can be different for each $\V{z}_i$.

We are interested in deriving an exponential decay concentration bound, which holds \textit{uniformly over parameter values} $\{\V{p}_i\}$, for the tail probability of $$  \frac{1}{n} \sum_{i=1}^n \left ( \sum_{k=1}^K   z_{ik} \log p_{ik} - \sum_{k=1}^K   p_{ik} \log p_{ik} \right ).$$ Specifically, let  $\mathcal{C}=\{\V{q}=(q_1,...,q_K): 0\leq q_k \leq 1, k=1,...,K, \sum_{k=1}^K q_k =1  \}$. We aim to derive a bound for
\begin{align}\label{main_interest}
\sup_{\V{p}_1\in \mathcal{C},...,\V{p}_n\in \mathcal{C}} \mathbb{P}\left ( \left | \frac{1}{n} \sum_{i=1}^n \left ( \sum_{k=1}^K   z_{ik} \log p_{ik} - \sum_{k=1}^K   p_{ik} \log p_{ik} \right ) \right | \geq  \epsilon \right ), 
\end{align}
where  ``sup'' is understood as taking the supremum over all possible values of $\V{p}_1,...,\V{p}_n$ in $\mathcal{C}^n$,  not the maximum of $n$ values. 

We now give the main theorem. 
\begin{theorem}[Main result]\label{thm:main}
	For  sufficiently small positive $\epsilon$ and $K\geq 2$, 
	\begin{align}
	\sup_{\V{p}_1\in \mathcal{C},...,\V{p}_n\in \mathcal{C}}\mathbb{P}\left ( \left | \frac{1}{n} \sum_{i=1}^n \left ( \sum_{k=1}^K   z_{ik} \log p_{ik} - \sum_{k=1}^K   p_{ik} \log p_{ik} \right ) \right | \geq  \epsilon \right )
	\leq   2 \exp \left ( -\frac{n\epsilon^2}{4 (\max\{\log K, \log 5\} )^2} \right ). \label{main1}
	\end{align}
	Furthermore, if $(\log K)^2 /n =\Omega(1)$, for all $\epsilon>0$,
	\begin{align}
	\sup_{\V{p}_1\in \mathcal{C},...,\V{p}_n\in \mathcal{C}} \mathbb{P}\left ( \left | \frac{1}{n} \sum_{i=1}^n \left ( \sum_{k=1}^K   z_{ik} \log p_{ik} - \sum_{k=1}^K   p_{ik} \log p_{ik} \right ) \right | \geq  \epsilon \right ) 
	\not \rightarrow 0. \label{main2}
	\end{align}
\end{theorem}
Most results in the paper in fact hold for $K=1$. The corresponding log-likelihood is however non-random so we exclude this trivial case. 

We comment on the contributions of the main theorem before proceeding. Firstly, as aforementioned, inequality \eqref{main1} is uniform over parameter values as the right hand side does not depend on $\V{p}_1,...,\V{p}_n$. Note that we do not assume $\V{p}_1,...,\V{p}_n$ are bounded away\footnote{We use the convention $0 \log 0=0$, which is consistent with the limit $\lim_{q \rightarrow 0} q \log q=0$.} from the boundaries.  Removing this restriction is a challenge and a significant contribution of this paper. 
Secondly, inequality \eqref{main1} implies that if $(\log K)^2/n =o(1)$, for all $\epsilon>0$, 
\begin{align}
\sup_{\V{p}_1\in \mathcal{C},...,\V{p}_n\in \mathcal{C}} \mathbb{P}\left ( \left | \frac{1}{n} \sum_{i=1}^n \left ( \sum_{k=1}^K   z_{ik} \log p_{ik} - \sum_{k=1}^K   p_{ik} \log p_{ik} \right ) \right | \geq  \epsilon \right ) \label{main3}
\rightarrow 0.
\end{align}
The rate $(\log K)^2 /n=o(1)$ falls into the high-dimensional setting -- that is, the number of parameters can grow much faster than the sample size. Thirdly, \eqref{main2} and \eqref{main3} imply that the rate $(\log K)^2 /n= o(1)$ is optimal in the asymptotic sense. 

\section{Motivation and related work} \label{sec:motivation}
Consider the most classical case where $K$ is fixed and  $\V{z}_1,...,\V{z}_n$ are independently and identically distributed (i.i.d.) with $\V{p}_1=\cdots=\V{p}_n=\V{p}$.  By the law of large numbers, 
\begin{align}\label{aep}
\frac{1}{n} \sum_{i=1}^n  \sum_{k=1}^K   z_{ik} \log p_{k}  \stackrel{p}{\rightarrow}  \sum_{k=1}^K   p_{k} \log p_{k}, \,\, \mbox{as } n \rightarrow \infty.
\end{align}
This  result, called the asymptotic equipartition property (AEP), is one of the most classical results in information theory \cite{shannon1948mathematical}. The AEP has been generalized to stationary ergodic processes  \cite{mcmillan1953basic,breiman1957individual} and is referred to as the Shannon-McMillan-Breiman theorem. 

We improve the AEP for independent variables from a different perspective. We aim to prove a  non-asymptotic concentration inequality, more specifically, an exponential decay bound, for the tail probability. The study of exponential decay concentration inequalities for sums of binary variables dates back to at least the 1920s \cite{Bernstein1920}. The Chernoff-Hoeffding theorem \cite{Hoeffding63} gives the sharpest bound that can be derived by the Chernoff bound technique for sums of independent Bernoulli variables. Bernstein's inequality and Hoeffding's inequality \cite{Hoeffding63}  give concentration bounds with more tractable forms and can be generalized to bounded variables and more general settings, such as sub-Gaussian variables and sub-exponential variables. 

Most of these studies focused on tail bounds for sums of  variables. There is a lack of research on concentration inequalities for \textit{log-likelihoods}, i.e., sums weighted by  logarithms of the parameters. Uniform bounds that are independent of parameter values are particularly under-explored, despite their applications in  statistics \cite{choi2012stochastic,paul2016consistent,zhao2020identifiability}. 

We first discuss the difficulty in classical results when applied to log-likelihoods and then explain the motivation of the present research. For $K=2$, $z_{i2}=1-z_{i1}$ and $p_{i2}=1-p_{i1}$,  $i=1,...,n$. Then 
$\sum_{k=1}^2  z_{ik} \log p_{ik} - \sum_{k=1}^2  p_{ik} \log p_{ik}=(z_{i1}-p_{i1})\log \frac{p_{i1}}{1-p_{i1}}$. Assume $|\log \frac{p_{i1}}{1-p_{i1}}|\leq M, i=1,...,n$. By Bernstein's inequality (see \cite{dubhashi2009concentration}, Theorem 1.2),
for all $\epsilon>0$, 
\begin{align}
\mathbb{P}\left ( \left | \frac{1}{n} \sum_{i=1}^n \left ( \sum_{k=1}^2  z_{ik} \log p_{ik} - \sum_{k=1}^2  p_{ik} \log p_{ik} \right ) \right | \geq  \epsilon \right ) \leq 2 \exp \left \{ -\frac{n^2 \epsilon^2/2 }{ \sum_{i=1}^n \textnormal{Var}(z_{i1}\log \frac{p_{i1}}{1-p_{i1}}) + M n \epsilon/3}    \right \}. \label{Bern}
\end{align}
The reader is referred to \cite{choi2012stochastic} for an application of \eqref{Bern} to community detection in networks. A drawback of \eqref{Bern} is that the condition  $|\log \frac{p_{i1}}{1-p_{i1}}|\leq M$ requires $p_{i1}$ to be bounded away from 0 and 1. Otherwise, the bound can become trivial if $M$ grows with $n$ too fast. One may apply alternative forms of Bernstein's inequality (for example, Theorem 2.8.2 in \cite{vershynin2018high}) or other commonly-used concentration inequalities to log-likelihoods and faces a similar problem. 

The essential problem is that $\{\log p_{ik} \}$ should not be treated as an arbitrary set of coefficients because $p_{ik}$ is also a part of the model that controls the probabilistic behavior of $z_{ik}$. 
To the best of our knowledge, Zhao \cite{zhao2020note} first overcame this technical difficulty. The paper removed the constraint $|\log \frac{p_{i1}}{1-p_{i1}}|\leq M$ and proved a new Bernstein-type bound that does not depend on $\{\V{p}_{i}\}$: 
\begin{theorem}[\cite{zhao2020note}, Corollary 1]\label{thm:old}
	For  $K\geq 2$ and  $\epsilon>0$,
	\begin{align*}
	\sup_{\V{p}_1\in \mathcal{C},...,\V{p}_n\in \mathcal{C}} \mathbb{P}\left ( \left | \frac{1}{n} \sum_{i=1}^n \left ( \sum_{k=1}^K   z_{ik} \log p_{ik} - \sum_{k=1}^K   p_{ik} \log p_{ik} \right ) \right | \geq  \epsilon \right )  \leq 2K \exp \left \{ -\frac{n \epsilon^2}{2K(K+\epsilon)} \right \}.
	\end{align*}
\end{theorem}

The above theorem implies that if $ (K^2\log K)/n \rightarrow 0$, for all $\epsilon$,
\begin{align*}
\sup_{\V{p}_1\in \mathcal{C},...,\V{p}_n\in \mathcal{C}} \mathbb{P}\left ( \left | \frac{1}{n} \sum_{i=1}^n \left ( \sum_{k=1}^K   z_{ik} \log p_{ik} - \sum_{k=1}^K   p_{ik} \log p_{ik} \right ) \right | \geq  \epsilon \right )
\rightarrow 0.
\end{align*} 
Our goal is to improve the above rate. We approach this problem by first considering an elementary probability inequality -- Chebyshev's inequality: 
\begin{align}
\mathbb{P}\left ( \left | \frac{1}{n} \sum_{i=1}^n \left ( \sum_{k=1}^K   z_{ik} \log p_{ik} - \sum_{k=1}^K   p_{ik} \log p_{ik} \right ) \right | \geq  \epsilon \right ) 
\leq  \frac{\sum_{i=1}^n\textnormal{Var}( L(\V{z}_i))}{n^2\epsilon^2}, \label{cheby}
\end{align}
where 
\begin{align*}
\textnormal{Var}( L(\V{z}_i))=\sum_{k=1}^K p_{ik} (\log p_{ik})^2- \left ( \sum_{k=1}^K p_{ik} \log p_{ik} \right )^2 \leq \sum_{k=1}^K p_{ik} (\log p_{ik})^2 .
\end{align*}
By taking the first and the second derivatives of $p_{ik}(\log p_{ik})^2$, one can easily prove $\max_{p_{ik}\in [0,1]} p_{ik}(\log p_{ik})^2 =4 e^{-2}$.
We can therefore give a rough estimate of the right hand side of \eqref{cheby}: 
\begin{align*}
\frac{\sum_{i=1}^n \textnormal{Var}( L(\V{z}_i))}{n^2\epsilon^2} \leq \frac{4K}{n\epsilon^2 e^2 }.
\end{align*}
First note that the above bound is independent of $\{\V{p}_{i}\}$, which is in line with the observation in \cite{zhao2020note}. Moreover, the bound implies that  if $ K =o(n)$, for all $\epsilon>0$,
\begin{align*}
\sup_{\V{p}_1\in \mathcal{C},...,\V{p}_n\in \mathcal{C}} \mathbb{P}\left ( \left | \frac{1}{n} \sum_{i=1}^n \left ( \sum_{k=1}^K   z_{ik} \log p_{ik} - \sum_{k=1}^K   p_{ik} \log p_{ik} \right ) \right | \geq  \epsilon \right )
\rightarrow 0,
\end{align*} 
which clearly suggests that there is room for improvement in Theorem 2 when $K$ is large. 

The above estimate of $\textnormal{Var}( L(\V{z}_i))$ is rough because it ignores the constraint $\sum_{k=1}^K p_{ik}=1$. In Section \ref{sec:main}, we will show that the correct order of $\textnormal{Var}( L(\V{z}_i))$ is  $(\log K)^2$. 

The rest of the paper is organized as follows. In Section \ref{sec:fixed}, we prove the concentration inequalities for fixed parameters $\V{p}_1,...,\V{p}_n$ by classical techniques. In Section \ref{sec:main}, we elaborate the main theorem and prove it through a series of lemmas and theorems. Our approach relies on bounding certain moment generating functions (MGFs).  To bound the  MGFs, we borrow the idea of primal and dual from the literature of optimization.  In Section \ref{sec:extension}, we extend the results to misspecified log-likelihoods for grouped random variables. In Section \ref{sec:application}, we give two examples of the applications of the new result in information theory:  a refined explicit bound for $n$ in Shannon's source coding theorem and the error exponent for  source coding  with growing $K$. 

\section{Inequalities for fixed parameters} \label{sec:fixed}
We prove concentration inequalities for fixed $\V{p}_1,...,\V{p}_n$ in this section, where each $\V{p}_i$ is an interior point of $\mathcal{C}$. The proofs are not challenging. But the results do not exist in the literature as the form we present below, to the best of our knowledge. So we include them for completeness. Moreover, the proofs shed light on the asymmetry of the two sides of the bound and the challenge in proving the uniform bound.  

Let $Y_i=\sum_{k=1}^K   z_{ik} \log p_{ik} - \sum_{k=1}^K   p_{ik} \log p_{ik}$. Let $M_Y(\lambda,\V{p}_i)=\mathbb{E}[e^{\lambda Y_i}]$ be the MGF\footnote{Note that the MGF is well defined on the boundaries of $\mathcal{C}$ for $\lambda>-1$. Specifically, $0^{\lambda+1}=0$ for $\lambda>-1$, which is continuous at 0 because $\lim_{p\rightarrow 0} p^{\lambda+1}=0$ for $\lambda>-1$.} of $Y_i$. 
\begin{theorem}[Right-tail bound for fixed parameters]\label{thm:positive}
	For $K\geq 2$ and $\lambda>0$,\begin{align}
	M_Y(\lambda,\V{p}_i)  \leq \exp \left (  \lambda^2 \sum_{k=1}^K p_{ik} (\log p_{ik})^2 \right ), \,\, i=1,...,n. \label{right_mgf}
	\end{align}
	For $K\geq 2$ and $\epsilon>0$,
	\begin{align*}	
	\mathbb{P}\left ( \frac{1}{n} \sum_{i=1}^n \left ( \sum_{k=1}^K   z_{ik} \log p_{ik} - \sum_{k=1}^K   p_{ik} \log p_{ik} \right ) \geq  \epsilon \right ) \leq \exp \left ( -\frac{n^2 \epsilon^2}{4 \sum_{i=1}^n \sum_{k=1}^K p_{ik} (\log p_{ik})^2 } \right ).
	\end{align*}	
\end{theorem}
\begin{proof}
	First note two elementary inequalities: $\log(1+x)\leq x$ for $x>-1$, and $e^x\leq 1+x+x^2$ for $x\leq 1$.  For $\lambda>0$,
	\begin{align}
	\log M_Y(\lambda,\V{p}_i)  = & \log \left (  \sum_{k=1}^K p_{ik}^{\lambda+1} \right )-\lambda \sum_{k=1}^K p_{ik} \log p_{ik} \nonumber \\
	\leq  & \sum_{k=1}^K p_{ik}^{\lambda+1}-1-\lambda \sum_{k=1}^K p_{ik} \log p_{ik} \nonumber \\
	= & \sum_{k=1}^K p_{ik} \exp(\lambda \log p_{ik})-1-\lambda \sum_{k=1}^K p_{ik} \log p_{ik} \nonumber \\
	\leq & \sum_{k=1}^K p_{ik} \left (1+\lambda \log p_{ik} +\lambda^2 (\log p_{ik})^2 \right )-1-\lambda \sum_{k=1}^K p_{ik} \log p_{ik} = \lambda^2 \sum_{k=1}^K p_{ik} (\log p_{ik})^2. \label{positive}
	\end{align}
	Therefore, $
	M_Y(\lambda,\V{p}_i)  \leq \exp \left (  \lambda^2 \sum_{k=1}^K p_{ik} (\log p_{ik})^2 \right ).
	$
	
	The rest of the proof follows from a standard Chernoff bound argument on \textit{sub-Gaussian} variables (for example, see Chapter 2 of \cite{wainwright_2019}). For $\lambda>0$, by Markov's inequality,
	\begin{align}
	& \mathbb{P} \left (     \sum_{i=1}^n Y_i  \geq n\epsilon  \right ) = \mathbb{P} \left ( e^{\lambda \sum_{i=1}^n Y_i}  \geq e^{\lambda n \epsilon }  \right ) \leq \frac{\prod_{i=1}^n \mathbb{E} [e^{\lambda Y_i}]}{ e^{\lambda n \epsilon } } \leq \exp \left \{\lambda^2 \sum_{i=1}^n \sum_{k=1}^K p_{ik} (\log p_{ik})^2 -\lambda n \epsilon \right  \}. \label{chernoff} 
	\end{align}
	We obtain the result by letting $\lambda=\frac{n \epsilon}{2\sum_{i=1}^n \sum_{k=1}^K p_{ik} (\log p_{ik})^2} $.
\end{proof}

\begin{theorem}[Left-tail bound for fixed parameters]\label{thm:negative}
	Let $b=\max_{i=1,...,n,k=1,...,K} |\log p_{ik}|$. 		For $K\geq 2$ and $-1/b\leq \lambda<0$,
	\begin{align}
	M_Y(\lambda,\V{p}_i)  \leq \exp \left (  \lambda^2 \sum_{k=1}^K p_{ik} (\log p_{ik})^2 \right ), \,\, i=1,...,n. \label{left_mgf}
	\end{align}
	For $K\geq 2$ and $\epsilon>0$,
	\begin{align}	
	& \mathbb{P}\left ( \frac{1}{n} \sum_{i=1}^n \left ( \sum_{k=1}^K   z_{ik} \log p_{ik} - \sum_{k=1}^K   p_{ik} \log p_{ik} \right ) \leq -\epsilon \right ) \nonumber \\
	\leq &  \left \{ \begin{array}{ll}
	\exp \left ( -\frac{n^2\epsilon^2}{4 \sum_{i=1}^n \sum_{k=1}^K p_{ik} (\log p_{ik})^2 } \right ) & \text{for } 0<\epsilon \leq \frac{2\sum_{i=1}^n \sum_{k=1}^K p_{ik} (\log p_{ik})^2}{n b} \\
	\exp \left ( -\frac{n\epsilon}{2 b} \right ) & \text{for } \epsilon > \frac{2\sum_{i=1}^n \sum_{k=1}^K p_{ik} (\log p_{ik})^2}{n b}.
	\end{array}
	\right . 
	\end{align}
\end{theorem}
\begin{proof}
	For $- |\frac{1}{\log p_{ik}}|\leq \lambda<0,i=1,...,n,k=1,...,K$,
	\begin{align}
	\log M_Y(\lambda,\V{p}_i)  
	\leq & \sum_{k=1}^K p_{ik} \exp(\lambda \log p_{ik})-1-\lambda \sum_{k=1}^K p_{ik} \log p_{ik} \nonumber \\
	\leq & \sum_{k=1}^K p_{ik} \left (1+\lambda \log p_{ik} +\lambda^2 (\log p_{ik})^2 \right )-1-\lambda \sum_{k=1}^K p_{ik} \log p_{ik} =\lambda^2 \sum_{k=1}^K p_{ik} (\log p_{ik})^2, \label{negative}
	\end{align}
	the second inequality follows from $\lambda \log p_{ik}\leq 1$ and $e^x\leq 1+x+x^2$ for $x\leq 1$. 
	Therefore, $
	M_Y(\lambda,\V{p}_i)  \leq \exp \left (  \lambda^2 \sum_{k=1}^K p_{ik} (\log p_{ik})^2 \right )
	$, $i=1,...,n$, for  $-1/b\leq \lambda<0$.
	
	The rest of the proof follows from a standard Chernoff bound argument on \textit{sub-exponential} variables. For $-1/b \leq \lambda<0$, by a similar argument in \eqref{chernoff},
	\begin{align*}
	& \mathbb{P} \left (     \sum_{i=1}^n Y_i  \leq -n\epsilon  \right ) = \mathbb{P} \left ( e^{\lambda \sum_{i=1}^n Y_i}  \geq e^{-\lambda n \epsilon }  \right ) \leq \frac{\prod_{i=1}^n \mathbb{E} [e^{\lambda Y_i}]}{ e^{-\lambda n \epsilon } } \leq \exp \left \{\lambda^2 \sum_{i=1}^n \sum_{k=1}^K p_{ik} (\log p_{ik})^2 +\lambda n \epsilon \right  \}.  \\
	\end{align*}
	We obtain the result by letting $\lambda=-\frac{n \epsilon}{2\sum_{i=1}^n \sum_{k=1}^K p_{ik} (\log p_{ik})^2}$ for $0< \epsilon \leq \frac{2\sum_{i=1}^n \sum_{k=1}^K p_{ik} (\log p_{ik})^2}{n b}$ and $\lambda=-1/b$ for $\epsilon > \frac{2\sum_{i=1}^n \sum_{k=1}^K p_{ik} (\log p_{ik})^2}{n b}$.
\end{proof}

A key difference between the two results is in the second inequality of \eqref{positive} and \eqref{negative}, where  $\lambda$ in the two theorems have opposite signs. The inequality in \eqref{positive} always holds  since $\lambda \log p_{ik}\leq 0$, 
but the inequality in \eqref{negative} cannot be true for a large $\lambda \log p_{ik}$ when $\lambda<0$ because of the exponential growth. Therefore, it is not difficult to find a quantity that does not depend on $\{\V{p}_i\}$ to bound the right-tail probability according to the discussion in Section \ref{sec:motivation} (obtaining the optimal order $(\log K)^2$ is however nontrivial and will be shown in Section \ref{sec:main}). To find a uniform bound for the left-tail probability is, however, more challenging since a positive $b$ does not exist in Theorem \ref{thm:negative} if $p_{ik} \rightarrow 0$ for some $i,k$. We develop a new technique to uniformly control $\lambda \log p_{ik}$ when $\lambda<0$ in Section \ref{sec:main}.

The asymmetry of the left and right tails can be understood by the following heuristic argument. Note that $p_{ik}$ can only contribute to the positive part of $\sum_{k=1}^K   z_{ik} \log p_{ik} - \sum_{k=1}^K   p_{ik} \log p_{ik}$ when $z_{ik}=0$. The contribution  $-p_{ik}\log p_{ik}$ is however negligible when $p_{ik}$ is close to 0. On the other hand,  $p_{ik}$ contributes to the negative part of $\sum_{k=1}^K   z_{ik} \log p_{ik} - \sum_{k=1}^K   p_{ik} \log p_{ik}$ when $z_{ik}=1$, and the contribution $(1-p_{ik})\log p_{ik}$ blows up when $p_{ik}$ is close to 0.

\section{Proof of the main result} \label{sec:main}
We break up Theorem \ref{thm:main} into a number of intermediate  results. Firstly, we prove the uniform convergence of \eqref{main_interest} under the condition  $ (\log K)^2 /n=o(1)$ by establishing a polynomial decay bound for \eqref{main_interest}. Secondly, we prove that \eqref{main_interest} does not converge to 0 if $(\log K)^2 /n=\Omega(1)$, which implies   $ (\log K)^2 /n=o(1)$ is the optimal rate. Finally, we prove the most difficult part, i.e., the exponential decay bound for \eqref{main_interest}.  

Recall that $\mathcal{C}=\{\V{q}=(q_1,...,q_K): 0\leq q_k \leq 1, k=1,...,K, \sum_{k=1}^K q_k =1  \}$, which gives the constraints each $\V{p}_i$ must satisfy. Let $\mathcal{D}=\{\V{q}=(q_1,...,q_K): 0\leq q_k \leq 1, k=1,...,K \}$ be another domain which excludes the constraint $\sum_{k=1}^K q_k =1$.

We begin by a lemma on the upper bound for $\textnormal{Var}( L(\V{z}_i))$, which may be of independent interest. Below we omit the index $i$ since the result is independent of $i$.   
\begin{lemma}\label{control_var}
	For $K\geq 5$, 
	\begin{align*}
	\max_{\V{p} \in \mathcal{C}} \sum_{k=1}^K p_{k} (\log p_{k})^2 =(\log K)^2.
	\end{align*}
\end{lemma}
\begin{proof}
	The statement in the lemma is equivalent to the following optimization problem: 
	\begin{align*}
	& \max_{\V{p} \in\mathcal{D}} \sum_{k=1}^K p_{k} (\log p_{k})^2,  \\
	& \textnormal{subject to }  \sum_{k=1}^K p_{k} =1. \nonumber
	\end{align*}
	Consider the Lagrangian function \cite{boyd2004convex}:
	\begin{align*}
	\mathscr{L}(\V{p}, \nu)= \sum_{k=1}^K p_{k} (\log p_{k})^2+\nu \left (\sum_{k=1}^K p_{k} -1 \right ).
	\end{align*}
	Define $
	g(\nu)=\max_{\V{p} \in \mathcal{D}} \mathscr{L}(\V{p}, \nu) 
	$.	Since $\mathcal{C}\subset \mathcal{D}$, $g(\nu)\geq \mathscr{L}(\tilde{\V{p}}, \nu)$ for all $\nu \in \mathbb{R}$ and $\tilde{\V{p}} \in \mathcal{C}$.
	Furthermore, by noting that for all  $\nu \in \mathbb{R}$ and $\tilde{\V{p}} \in \mathcal{C}$,
	\begin{align*}
	\mathscr{L}(\tilde{\V{p}}, \nu)=  \sum_{k=1}^K \tilde{p}_{k} (\log \tilde{p}_{k})^2+\nu \left (\sum_{k=1}^K \tilde{p}_{k} -1 \right )=\sum_{k=1}^K \tilde{p}_{k} (\log \tilde{p}_{k})^2,
	\end{align*}
	we have $g(\nu)\geq \sum_{k=1}^K \tilde{p}_{k} (\log \tilde{p}_{k})^2$ for $\tilde{\V{p}} \in \mathcal{C}$, which further implies $g(\nu)\geq \max_{\tilde{\V{p}} \in \mathcal{C}} \sum_{k=1}^K \tilde{p}_{k} (\log \tilde{p}_{k})^2 $. 
	
	The argument above shows that for all $\nu \in \mathbb{R}$, $g(\nu)$ is an upper bound for the original problem $\max_{\V{p} \in \mathcal{C}} \sum_{k=1}^K p_{k} (\log p_{k})^2$. Below we pick  $\nu=-(\log K)^2+2\log K$.
	Note that the optimization problem 
	\begin{align*}
	\max_{\V{p} \in \mathcal{D}} \sum_{k=1}^K p_{k} (\log p_{k})^2+\nu \left (\sum_{k=1}^K p_{k} -1 \right )
	\end{align*}
	is equivalent to $K$ separate problems: for $k=1,...,K$,
	\begin{align*}
	\max_{0\leq p_{k} \leq 1}  h(p_{k}):=p_{k} (\log p_{k})^2+\nu  p_{k}.
	\end{align*}	
	By taking the derivative with respect $p_{k}$, the local maximizer satisfies
	\begin{align*}
	(\log p_{k})^2+2 \log p_{k}+\nu=0.
	\end{align*}
	There are two candidate solutions of the quadratic equation $y^2+2y+\nu=0$:
	\begin{align*}
	y=-1-\sqrt{1-\nu}, y=-1+\sqrt{1-\nu},
	\end{align*}
	which are
	\begin{align*}
	y=-1-(\log K-1), y=-1+(\log K -1).
	\end{align*}	
	The corresponding solutions of $p_{k}$ are 
	\begin{align*}
	p_{k}=1/K, p_{k}=\exp(\log K-2).
	\end{align*}	
	Because $h''(1/K)<0$ and $h''(\exp(\log K-2))>0$ for $K\geq 3$, $1/K$ is the only local maximizer in $(0,1)$. Furthermore, because $h(1/K)\geq h(0)$ and $h(1/K) \geq h(1)$ for $K\geq 5$, $1/K$ is the global maximizer.
	
	Therefore,
	$
	g(\nu)=(\log K)^2  \geq \max_{\V{p} \in \mathcal{C}} 	 \sum_{k=1}^K p_{k} (\log p_{k})^2
	$. 
	The equality holds  because $(1/K,...,1/K)  \in \mathcal{C}$.
\end{proof}

Lemma \ref{control_var} can be viewed as a second-order version of a well-known inequality for entropies: $-\sum_{k=1}^K p_k (\log p_{k}) \leq \log K$, which can be proved by Jensen's inequality (see Theorem 2.6.4 in \cite{cover2006elements}).  But Lemma 1 is more difficult to prove because the function involved is neither convex nor concave. 

The next theorem immediately follows from Lemma \ref{control_var} and \eqref{cheby}.
\begin{theorem}[Uniform convergence]
	If $(\log K)^2 /n =o(1)$, for all $\epsilon>0$,
	\begin{align*}
	\sup_{\V{p}_1\in \mathcal{C},...,\V{p}_n\in \mathcal{C}} \mathbb{P}\left ( \left | \frac{1}{n} \sum_{i=1}^n \left ( \sum_{k=1}^K   z_{ik} \log p_{ik} - \sum_{k=1}^K   p_{ik} \log p_{ik} \right ) \right | \geq  \epsilon \right )
	\rightarrow 0.
	\end{align*} 
\end{theorem}

We now prove that $(\log K)^2 /n=o(1)$ is the optimal rate.

\begin{theorem}[Rate optimality]\label{thm:counter_eg}
	If $(\log K)^2 /n= \Omega(1)$,
	\begin{align*}
	\sup_{\V{p}_1 \in \mathcal{C},...,\V{p}_n\in \mathcal{C}} \textnormal{Var} \left(  \frac{1}{n} \sum_{i=1}^n  \sum_{k=1}^K   z_{ik} \log p_{ik} \right ) {\not \rightarrow} 0, 
	\end{align*}
	and for all $\epsilon>0$,
	\begin{align}
	\sup_{\V{p}_1 \in \mathcal{C},...,\V{p}_n\in \mathcal{C}} \mathbb{P}\left ( \left | \frac{1}{n} \sum_{i=1}^n \left ( \sum_{k=1}^K   z_{ik} \log p_{ik} - \sum_{k=1}^K   p_{ik} \log p_{ik} \right ) \right | \geq  \epsilon \right ) {\not \rightarrow} 0. \label{non_converge}
	\end{align}
\end{theorem}
\begin{proof}
	We only need to find a parameter setting $\{\V{p}_i^*\}$ such that $\{\V{z}^*_{i} \}$ generated under $\{\V{p}_i^*\}$ satisfy
	\begin{align*}
	\textnormal{Var} \left(  \frac{1}{n} \sum_{i=1}^n  \sum_{k=1}^K   z_{ik}^* \log p_{ik}^* \right ) {\not \rightarrow} 0, 
	\end{align*}
	and
	\begin{align*}
	\mathbb{P}\left ( \left | \frac{1}{n} \sum_{i=1}^n \left ( \sum_{k=1}^K   z_{ik}^* \log p_{ik}^* - \sum_{k=1}^K   p_{ik}^* \log p_{ik}^* \right ) \right | \geq  \epsilon \right ) {\not \rightarrow} 0.
	\end{align*}
	Let $p^*_{i1}=1/2, p^*_{i2}=p^*_{i3}=\cdots= p^*_{iK}=\frac{1}{2(K-1)}$ for $i=1,...,n$ and let $\{z^*_{ik} \}$ be the corresponding random variables generated under $\{\V{p}_i^*\}$.
	
	Let $w_1,...,w_n$ be i.i.d.  variables where 
	\begin{align*}
	w_i=  \left \{ \begin{array}{ll}
	1 & \text{w.p. } 1/2 \\
	-1 & \text{w.p. } 1/2.
	\end{array} \right .
	\end{align*}
	Then it is easy to check that
	\begin{align*}
	\sum_{i=1}^n \left ( \sum_{k=1}^K   z^*_{ik} \log p^*_{ik} - \sum_{k=1}^K   p^*_{ik} \log p^*_{ik} \right ) \stackrel{d}{=} \frac{1}{2} \log (K-1) \sum_{i=1}^n w_i.
	\end{align*}
	Therefore, 
	$
	\textnormal{Var} \left(  \frac{1}{n} \sum_{i=1}^n  \sum_{k=1}^K   z_{ik}^* \log p_{ik}^* \right ) = \frac{1}{4n}(\log (K-1))^2 {\not \rightarrow} 0
	$ if $(\log K)^2 /n= \Omega(1)$.
	
	By the Berry-Esseen theorem (see Theorem 3.4.17 in \cite{durrett2019probability} for example), for any $x \in (-\infty,\infty)$,
	\begin{align*}
	\left | \mathbb{P} \left ( \frac{1}{\sqrt n}\sum_{i=1}^n w_i \leq x \right )-\Phi(x) \right | \leq \frac{C}{\sqrt n},
	\end{align*}
	where $\Phi(\cdot)$ is the 
	cumulative distribution function of the standard normal distribution and $C$ is an absolute constant. Therefore, 
	\begin{align*}
	\mathbb{P}\left (  \frac{1}{n} \sum_{i=1}^n \left ( \sum_{k=1}^K   z^*_{ik} \log p^*_{ik} - \sum_{k=1}^K   p^*_{ik} \log p^*_{ik} \right )  \leq  -\epsilon \right ) \geq \Phi \left ( \frac{-2 \sqrt{n} \epsilon }{\log(K-1)} \right )-\frac{C}{\sqrt n},
	\end{align*}
	where 
	$
	\Phi \left ( \frac{-2 \sqrt{n} \epsilon }{\log(K-1)} \right ) \not \rightarrow 0
	$
	if $(\log K)^2 /n= \Omega(1)$.
\end{proof}

It is a useful idea to get a sense of the strongest possible concentration inequality by checking the convergence rate of the variance (see the introduction of \cite{vershynin2019concentration} for example).  It is also worth mentioning that $\mbox{Var}{(X_n)}\not \rightarrow 0$ does not automatically imply $\mathbb{P}{(|X_n-\mathbb{E}[X_n]|>\epsilon)}\not \rightarrow 0$ for an arbitrary sequence $\{X_n\}$. The result in \eqref{non_converge}  relies on tail behavior of the sum of independent variables. 

We now prove the exponential decay bound \eqref{main1}. The right-tail bound is  relatively easy to prove as pointed out in Section \ref{sec:fixed}. 
\begin{theorem}[Uniform bound for the right tail]\label{thm:unif_right}
	For $K\geq 5$ and $\lambda>0$,\begin{align}
	M_Y(\lambda,\V{p}_i)  \leq \exp \left (  \lambda^2 (\log K)^2 \right ), \,\, i=1,...,n. \label{inequality_MGF}
	\end{align}
	For $K\geq 5$ and $\epsilon>0$,
	\begin{align*}	
	\sup_{\V{p}_1\in \mathcal{C},...,\V{p}_n\in \mathcal{C}}  \mathbb{P}\left ( \frac{1}{n} \sum_{i=1}^n \left ( \sum_{k=1}^K   z_{ik} \log p_{ik} - \sum_{k=1}^K   p_{ik} \log p_{ik} \right ) \geq  \epsilon \right ) \leq \exp \left ( -\frac{n \epsilon^2}{4 (\log K)^2} \right ).
	\end{align*}	
\end{theorem}
\begin{proof}
	The first conclusion immediately follows from  \eqref{right_mgf} and Lemma \ref{control_var}.  By the standard Chernoff bound argument on sub-Gaussian variables as in Theorem \ref{thm:positive}, for all $\V{p}_1\in \mathcal{C},...,\V{p}_n\in \mathcal{C}$,
	\begin{align*}	
	\mathbb{P}\left ( \frac{1}{n} \sum_{i=1}^n \left ( \sum_{k=1}^K   z_{ik} \log p_{ik} - \sum_{k=1}^K   p_{ik} \log p_{ik} \right ) \geq  \epsilon \right ) \leq \exp \left ( -\frac{n \epsilon^2}{4 (\log K)^2} \right ).
	\end{align*}	
	The second conclusion follows immediately. 
\end{proof}

We now prove the exponential decay bound for the left tail, which is the most difficult part in the main theorem. Note that we cannot directly apply \eqref{left_mgf} and Lemma \ref{control_var} in this case  because $b\rightarrow 0$ if $p_{ik} \rightarrow 0$ for some $i,k$. We therefore need to find a different approach for bounding the MGF. 

The next lemma is an optimization result, which gives an upper bound of the MGF. 
\begin{lemma}\label{thm:opt2}
	Let $F(\lambda,\V{p})=\sum_{k=1}^K p_k^{\lambda+1}-1-\lambda \sum_{k=1}^K p_k \log p_k$. For $K\geq 5$, $\lambda>-1$ and $\lambda \geq \frac{2-2/K-\log K}{(1-1/K) \log K }$,
	\begin{align*}
	\max_{\V{p}\in \mathcal{C}} F(\lambda,\V{p})=\exp(-\lambda \log K)-1 +\lambda \log K.
	\end{align*}
\end{lemma}
\begin{proof}
	We only need to prove the case $\lambda \neq 0$, otherwise the result is trivial. 		Consider the Lagrangian 
	\begin{align*}
	\mathscr{L}(\V{p}, \lambda, \nu)=  \sum_{k=1}^K p_k^{\lambda+1}-1-\lambda \sum_{k=1}^K p_k \log p_k+\nu \left (\sum_{k=1}^K p_k -1 \right ).
	\end{align*}
	By the same argument in Lemma \ref{control_var}, for all $\tilde{\V{p}}\in \mathcal{C}$ and $\nu \in \mathbb{R}$,
	\begin{align*}
	F(\lambda,\tilde{\V{p}}) \leq \max_{\V{p} \in \mathcal{D}} 	\mathscr{L}(\V{p}, \lambda, \nu).
	\end{align*}
	Fix $\nu=-(\lambda+1)(1/K)^\lambda+\lambda(-\log K +1)$. Note that the optimization problem $\max_{\V{p} \in \mathcal{D}} 	\mathscr{L}(\V{p}, \lambda, \nu)$ can be written as $K$ separate problems:
	for $k=1,...,K$,
	\begin{align*}
	\max_{0\leq p_k \leq 1} f(p_k),
	\end{align*}
	where $f(p_k)=p_k^{\lambda+1}-\lambda  p_k \log p_k+\nu p_k$.
	Since the $K$ optimization problems are identical, below we omit the index $k$. The first and second derivatives of $ f(p)$ are 
	\begin{align*}
	f'(p)= & (\lambda+1)p^\lambda-\lambda(\log p+1)+\nu, \\
	f''(p)=& (\lambda+1)\lambda p^{\lambda-1}-\lambda p^{-1}. 
	\end{align*}
	The choice of $\nu$ makes $1/K$ is a stationary point of $f(p)$ in $(0,1)$ since $f'(1/K)=0$. Below we prove that $1/K$ is the global maximizer in $[0,1]$. We need to prove $f(1/K)\geq f(1)$, $f(1/K) \geq f(0)$, and $1/K$ is the only local maximizer. The proof of $f(1/K)\geq f(1)$ and $f(1/K) \geq f(0)$ involves tedious calculation, so we leave it to Lemma \ref{thm:tedious} in the appendix.
	Next we prove that $1/K$ is the only local maximizer.
	
	Note that  the unique solution to $f''(p)=0$ in $(0,\infty)$ is $\exp \left (\frac{1}{\lambda}  \log \frac{1}{\lambda+1}  \right )$ when $\lambda \neq 0$ and $\lambda>-1$. Furthermore, $f''(p)< 0$ if $p< \exp \left (\frac{1}{\lambda}  \log \frac{1}{\lambda+1}  \right )$ and $f''(p)> 0$ if $p>\exp \left (\frac{1}{\lambda}  \log \frac{1}{\lambda+1}  \right )$. We only show the case $\lambda<0$: 
	\begin{align*}
	& (\lambda+1)\lambda p^{\lambda-1}-\lambda p^{-1} > (<) \, 0 \\
	\Leftrightarrow & (\lambda+1) p^\lambda< (>) \, 1 \\
	\Leftrightarrow & \lambda \log p < (>) \log \left ( \frac{1}{\lambda+1} \right ) \\
	\Leftrightarrow & p> (<)\exp \left ( \frac{1}{\lambda} \log \left ( \frac{1}{\lambda+1} \right ) \right ).
	\end{align*}
	The case $\lambda>0$ is similar. 
	
	Firstly, we prove that $1/K$ is a local maximizer, i.e., $f''(1/K)<0$. If $f''(1/K)\geq 0$, then for all $p>1/K$, by the mean value theorem there exists $\xi \in (1/K, p)$ such that $f'(p)-f'(1/K)=f''(\xi)(p-1/K)>0$ since $f''(\xi)>0$.  Therefore, $f'(p)>0$ for all $p>1/K$. Furthermore, there exists  $p \in (1/K,1)$ such that $f(1)-f(1/K)=f'(p)(1-1/K)>0$, which contradicts $f(1/K)\geq f(1)$.
	
	Secondly, we prove that $1/K$ is the only local maximizer. That is, there is no other $p$ such that $f'(p)=0$ and $f''(p)\leq 0$. If such a $p$ exists, without loss of generality, assume $1/K<p$. By the mean value theorem, there exists $\xi \in (1/K, p)$ such that $f'(p)-f'(1/K)=f''(\xi)(p-1/K)<0$ since $f''(\xi)<0$, which contradicts $f'(p)=f'(1/K)=0$. 
	
	Therefore, $1/K$ is the global maximizer of $f(p)$ in $[0,1]$. This implies 
	$\max_{\V{p} \in \mathcal{C}} 	F(\lambda,\V{p}) \leq \exp(-\lambda \log K)-1 +\lambda \log K$. The equality holds  because $(1/K,...,1/K)  \in \mathcal{C}$.
\end{proof}

We now give the uniform bound for $M_Y(\lambda,\V{p}_i)$ on both sides -- the key result of this paper, which covers  \eqref{inequality_MGF}. 
\begin{theorem}[Uniform bound for the MGF]\label{thm:bound_mgf}
	For $K\geq 5$ and $\lambda \geq -\min \left \{ \frac{1}{\log K},  \frac{\log K+2/K-2}{(1-1/K) \log K } \right \}$,
	\begin{align}\label{final_mgf}
	M_Y(\lambda,\V{p}_i)  \leq \exp \left (  \lambda^2 (\log K)^2 \right ), \,\, i=1,...,n. 
	\end{align}
\end{theorem}
\begin{proof}
	For $\lambda\geq -\frac{1}{\log K}$ and $\lambda \geq -\frac{\log K+2/K-2}{(1-1/K) \log K }$,
	\begin{align*}
	\log M_Y(\lambda,\V{p}_i) 
	& \leq   \sum_{k=1}^K p_{ik}^{\lambda+1}-1-\lambda \sum_{k=1}^K p_{ik} \log p_{ik} \\
	& \leq  \exp(-\lambda \log K)-1 +\lambda \log K \\
	& \leq 1-\lambda \log K +\lambda^2 (\log K)^2-1 +\lambda \log K=\lambda^2 (\log K)^2,
	\end{align*}
	where the second inequality follows from Lemma \ref{thm:opt2} and the third inequality follows from $e^x\leq 1+x+x^2$, $x\leq 1$.
\end{proof}
To prove the left-tail bound, we only need to use \eqref{final_mgf} for negative $\lambda$. By the standard Chernoff bound argument for sub-exponential variables as in Theorem \ref{thm:negative}, we obtain:
\begin{theorem}[Uniform bound for the left tail]\label{thm:unif_left}
	Let $b^*=1/\min \left \{ \frac{1}{\log K}, \frac{\log K+2/K-2}{(1-1/K) \log K } \right \}$. For $K\geq 5$ and $\epsilon>0$,
	\begin{align*}
	&\sup_{\V{p}_1\in \mathcal{C},...,\V{p}_n \in \mathcal{C}} \mathbb{P}\left (  \frac{1}{n} \sum_{i=1}^n \left ( \sum_{k=1}^K   z_{ik} \log p_{ik} - \sum_{k=1}^K   p_{ik} \log p_{ik} \right ) \leq  -\epsilon \right ) \\
	\leq &  \left \{ \begin{array}{ll}
	\exp \left ( -\frac{n\epsilon^2}{4 (\log K)^2} \right ) & \text{for } 0< \epsilon \leq \frac{2(\log K)^2}{b^*} \\
	\exp \left ( -\frac{n\epsilon}{2 b^*} \right ) & \text{for } \epsilon > \frac{2(\log K)^2}{b^*}.
	\end{array}
	\right . 
	\end{align*}
\end{theorem}
We conclude this section by two comments. Firstly, note that $\exp \left ( -\frac{n\epsilon^2}{4 (\log K)^2} \right )$ is usually the bound to be used when $K$ is large because  $b^*=\log K$ for large $K$ and $\epsilon \leq 2 \log K$ for small $\epsilon$. Secondly, when applying the theorem for $K=2,3,4$, simply replace $K$ by 5 in the above statement because one can make $K=5$ by adding several empty categories. See the statement in Theorem \ref{thm:main}.

\section{Extension to a misspecified model}\label{sec:extension}

In this section, we extend the main result to a misspecified likelihood function, where $\{\V{z}_i\}$ are grouped into different classes and $\{\V{z}_i\}$ within the same classes share the same $\V{p}_i$ in the likelihood function. This is a setup that frequently appears in theoretical studies of community detection, for example, Theorem 2 in \cite{choi2012stochastic}, Theorem 2 in \cite{paul2016consistent}, and Theorem 2.2 and 3.2 in \cite{zhao2020identifiability}. The definitions and results in this section closely follow Section 3 of \cite{zhao2020note} and we provide  details for completeness. 

Let 
$\V{z}_{1}^{(1)},\V{z}_2^{(1)},...,\V{z}_{n_1}^{(1)},\V{z}_{1}^{(2)},\V{z}_{2}^{(2)},...,\V{z}_{n_2}^{(2)},...,\V{z}_{1}^{(I)},\V{z}_{2}^{(I)},...,\V{z}_{n_I}^{(I)}$ be independent categorical variables, where $\V{p}_{j}^{(i)}$ is the parameter for  $\V{z}_{j}^{(i)}$, that is,  $p_{jk}^{(i)}=\mathbb{P}(\V{z}_{jk}^{(i)}=1)$. As in the previous sections, $\V{p}_{j}^{(i)}$ can be different for each variable. 

Furthermore, let $\sum_{i=1}^I n_i=n$. Let $\bar{p}^{(i)}_k$ be the average probability for category $k$ within group $i$, i.e., $\bar{p}^{(i)}_k= \frac{1}{n_i} \sum_{j=1}^{n_i}p_{jk}^{(i)}$ for $i=1,...,I, k=1,...,K$. Let $\bar{\V{p}}^{(i)}=\left (\bar{p}^{(i)}_1,...,\bar{p}^{(i)}_K \right)$. The misspecified log-likelihood  is defined as
\begin{align*}
\sum_{i=1}^I \sum_{j=1}^{n_i}  \sum_{k=1}^K   z_{jk}^{(i)} \log \bar{p}^{(i)}_k.
\end{align*}
Note that here the probability distribution of $\V{z}_{j}^{(i)}$  remains the same as elsewhere in the paper. That is, each categorical variable has its own parameters. The likelihood is, however, misspecified because $\log \bar{p}^{(i)}_k$ is assumed the same for random variables in the same group. 

\begin{theorem}[Inequalities for the misspecified model] The following statements hold true: 
	\begin{enumerate}[label=(\roman*)]
		\item 	For $K\geq 5$ and $\epsilon>0$,
		\begin{align*}	
		\sup_{\V{p}_1\in \mathcal{C},...,\V{p}_n\in \mathcal{C}}  \mathbb{P}\left ( \frac{1}{n} \sum_{i=1}^I \sum_{j=1}^{n_i}  \left ( \sum_{k=1}^K   z_{jk}^{(i)} \log \bar{p}^{(i)}_k- \sum_{k=1}^K   p_{jk}^{(i)} \log \bar{p}^{(i)}_k \right ) \geq  \epsilon \right ) \leq \exp \left ( -\frac{n \epsilon^2}{4 (\log K)^2} \right ).
		\end{align*}	
		\item  Let $b^*=1/\min \left \{ \frac{1}{\log K}, \frac{\log K+2/K-2}{(1-1/K) \log K } \right \}$. For $K\geq 5$ and $\epsilon>0$,
		\begin{align*}
		&\sup_{\V{p}_1\in \mathcal{C},...,\V{p}_n\in \mathcal{C}}  \mathbb{P}\left ( \frac{1}{n} \sum_{i=1}^I \sum_{j=1}^{n_i}  \left ( \sum_{k=1}^K   z_{jk}^{(i)} \log \bar{p}^{(i)}_k- \sum_{k=1}^K   p_{jk}^{(i)} \log \bar{p}^{(i)}_k \right ) \leq  -\epsilon \right )\\
		\leq &  \left \{ \begin{array}{ll}
		\exp \left ( -\frac{n\epsilon^2}{4 (\log K)^2} \right ) & \text{for } 0< \epsilon \leq \frac{2(\log K)^2}{b^*} \\
		\exp \left ( -\frac{n\epsilon}{2 b^*} \right ) & \text{for } \epsilon > \frac{2(\log K)^2}{b^*}.
		\end{array}
		\right . 
		\end{align*}
		\item 	If $(\log K)^2 /n =\Omega(1)$, for all $\epsilon>0$,
		\begin{align*}	
		\sup_{\V{p}_1\in \mathcal{C},...,\V{p}_n\in \mathcal{C}}  \mathbb{P}\left ( \left | \frac{1}{n} \sum_{i=1}^I \sum_{j=1}^{n_i}  \left ( \sum_{k=1}^K   z_{jk}^{(i)} \log \bar{p}^{(i)}_k- \sum_{k=1}^K   p_{jk}^{(i)} \log \bar{p}^{(i)}_k \right ) \right | \geq  \epsilon \right ) \not \rightarrow 0.
		\end{align*}
	\end{enumerate}

\end{theorem}
\begin{proof}
	Let $$U^{(i)}=\sum_{j=1}^{n_i}  \left (\sum_{k=1}^K   z_{jk}^{(i)} \log \bar{p}^{(i)}_k- \sum_{k=1}^K   p_{jk}^{(i)} \log \bar{p}^{(i)}_k \right )=\sum_{j=1}^{n_i}  \left (\sum_{k=1}^K   z_{jk}^{(i)} \log \bar{p}^{(i)}_k- \sum_{k=1}^K   \bar{p}^{(i)}_k \log \bar{p}^{(i)}_k \right ).$$
	\begin{align*}
	\mathbb{E}\left [e^{\lambda U^{(i)}} \right ] & =\prod_{j=1}^{n_i} \left [   \left ( \sum_{k=1}^K p_{jk}^{(i)} \exp \left (\lambda \log \bar{p}^{(i)}_k \right ) \right ) \exp \left (-\lambda \sum_{k=1}^K  \bar{p}^{(i)}_k \log \bar{p}^{(i)}_k \right )\right ] \\	
	& \leq \left [\frac{1}{n_i} \sum_{j=1}^{n_i}  \left ( \sum_{k=1}^K p_{jk}^{(i)} \exp \left (\lambda \log \bar{p}^{(i)}_k \right ) \right ) \exp \left (-\lambda \sum_{k=1}^K  \bar{p}^{(i)}_k \log \bar{p}^{(i)}_k \right )\right ]^{n_i} \\
	& = \left [   \left ( \sum_{k=1}^K \bar{p}^{(i)}_k \exp \left (\lambda \log \bar{p}^{(i)}_k \right ) \right ) \exp \left (-\lambda \sum_{k=1}^K  \bar{p}^{(i)}_k \log \bar{p}^{(i)}_k \right )\right ]^{n_i} \\ 
	& =[M_Y(\lambda,\bar{\V{p}}^{(i)})]^{n_i},
	\end{align*}
	where the inequality follows from the inequality of arithmetic and geometric means: $ \sqrt[n]{\prod_{i=1}^n a_i} \leq \sum_{i=1}^n a_i/ n$ for non-negative $a_1,...,a_n$.
	
	By Theorem  \ref{thm:bound_mgf}, for $K\geq 5$ and $\lambda \geq -\min \left \{ \frac{1}{\log K},  \frac{\log K+2/K-2}{(1-1/K) \log K } \right \}$, 
	\begin{align*}
	[M_Y(\lambda,\bar{\V{p}}^{(i)})]^{n_i}\leq \exp \left (  n_i \lambda^2 (\log K)^2 \right ),
	\end{align*}
	and
	\begin{align*}
	\prod_{i=1}^I \mathbb{E}\left [e^{\lambda U^{(i)}} \right ] \leq \exp \left ( \sum_{i=1}^I n_i \lambda^2 (\log K)^2 \right )=\exp \left (n \lambda^2 (\log K)^2 \right ).
	\end{align*}
	Statement (\romannumeral 1) and (\romannumeral 2) follow immediately. The proof of (\romannumeral 3) is identical to Theorem \ref{thm:counter_eg}.
\end{proof}

\section{Applications in information theory}\label{sec:application}
As mentioned in Section \ref{sec:motivation}, our main result is a refinement of \eqref{aep},  called the asymptotic equipartition property (AEP) in information theory. The AEP is the foundation
of many important results in this field \cite{cover2006elements}. We give two examples of the applications of the new result in information theory:  we prove a refined explicit bound for sample size $n$ in Shannon's source coding theorem, which is a fundamental result in information theory, and derive the error exponent for source coding  in the high-dimensional setting. 

Before proceeding with applications, we restate Theorem \ref{thm:main} in terms of ``$\log_2$'' because \textit{bit} is usually used as the unit of information in practice. The proof is trivial by replacing $\epsilon$ with $\epsilon \log 2 $ in Theorem \ref{thm:unif_right} and \ref{thm:unif_left}.
\begin{corollary}\label{thm:bit}
	For $K\geq 5$ and $\epsilon>0$,
	\begin{align}	
	\sup_{\V{p}_1\in \mathcal{C},...,\V{p}_n\in \mathcal{C}} \mathbb{P}\left (  \frac{1}{n} \sum_{i=1}^n \left ( \sum_{k=1}^K   z_{ik} \log_2 p_{ik} - \sum_{k=1}^K   p_{ik} \log_2 p_{ik} \right )  \geq  \epsilon \right ) \leq \exp \left ( -\frac{n \epsilon^2}{4 (\log_2 K)^2} \right ). \label{bit_right}
	\end{align}	
	For $K\geq 5$ and sufficiently small positive $\epsilon$,
	\begin{align}
	&\sup_{\V{p}_1\in \mathcal{C},...,\V{p}_n\in \mathcal{C}} \mathbb{P}\left ( \frac{1}{n} \sum_{i=1}^n \left ( \sum_{k=1}^K   z_{ik} \log_2 p_{ik} - \sum_{k=1}^K   p_{ik} \log_2 p_{ik} \right )  \leq  -\epsilon \right ) \leq \exp \left ( -\frac{n \epsilon^2}{4 (\log_2 K)^2} \right ). \label{bit_left}
	\end{align}
\end{corollary}

Recall that $X$ is a discrete random variable taking values in a finite alphabet $\mathcal{X}=\{x_1,...,x_K\}$ with $\V{p}=\{p_1,...,p_K\}$. Let $\V{X}^{n}=(X_1,...,X_n)$ be i.i.d. copies of $X$ and $\V{x}^n$ be a realization in $\mathcal{X}^n$. For the simplicity of notation, we only consider the i.i.d. case in this section. An $n$-to-$m$ binary block code \cite{csiszar2011information} consists a pair of coder and decoder
\begin{align*}
f: \mathcal{X}^n \rightarrow \{0,1\}^m, \,\, \varphi: \{0,1\}^m \rightarrow \mathcal{X}^n,
\end{align*}
where $f$ maps each $\V{x}^n$ to an $m$-length 0-1 sequence, and $\varphi$ maps each $m$-length 0-1 sequence to a certain $\V{x}^n$. 
If one requires $(f,\varphi)$ to be error-free, i.e., $f$ to be injective on $\mathcal{X}^n$, then clearly $m$ should be at least $\lceil n \log_2 K  \rceil$, where $\lceil \cdot \rceil$ means rounding up to the next integer. But if one can tolerate an arbitrarily small error, Shannon showed that essentially  $n H$ bits are needed in his foundational paper \cite{shannon1948mathematical}, where $H=-\sum_{k=1}^K p_k \log_2 p_k$. The result is called Shannon's source coding theorem. Here we follow the version in \cite{mackay2003information}. 

Let $S_\delta^n$ be the smallest subset of $\mathcal{X}^n$ satisfying: 
\begin{align*}
\mathbb{P}( \V{X}^n \in S_\delta^n) \geq 1-\delta,
\end{align*}
and the essential bit content of $\mathcal{X}^n$ is defined as
\begin{align*}
H_\delta(\mathcal{X}^n) = \log_2 |S_\delta^n|.
\end{align*}
\begin{theorem} [Shannon's source coding theorem (\cite{mackay2003information}, Theorem 4.1)]\label{thm:source_orginal}
	For  $0<\delta<1$ and $\epsilon>0$, there exists a positive integer $n_0$ such that for $n>n_0$, 
	\begin{align}
	\left | \frac{1}{n} H_\delta(\mathcal{X}^n)-H   \right | < \epsilon. \label{essential_bit}
	\end{align}
\end{theorem}
The key ingredient of the proof in \cite{mackay2003information} relies on Chebyshev's inequality \eqref{cheby}. We show in the next theorem how to apply the exponential decay bound in Corollary \ref{thm:bit} to derive an improved lower bound for $n$, which does not depends on $\V{p}$. %which increases logarithmically as the error $\delta$ decreases.
\begin{theorem} [Refinement of Shannon's source coding theorem]\label{thm:refine}
	For $K\geq 5$, $0<\delta<1$, and sufficiently small positive $\epsilon$,
	\begin{align*}
	\frac{1}{n} H_\delta(\mathcal{X}^n)  <H+\epsilon, & \textnormal{ for } n>\frac{4 (\log_2 K)^2 \log (1/\delta)}{\epsilon^2}, \\
	\frac{1}{n} H_\delta(\mathcal{X}^n)  >H-\epsilon, & \textnormal{ for }  n> \max \left \{ \frac{2\log_2(2/(1-\delta))}{\epsilon} ,\frac{16 (\log_2 K)^2 \log (2/(1-\delta))}{\epsilon^2} \right \}. 
	\end{align*}
\end{theorem}
\begin{proof}
	Define $T^n_{1,\epsilon}$ as 
	\begin{align*}
	T^n_{1,\epsilon} = \left \{ \V{x}^n \in \mathcal{X}^n:   \frac{1}{n} \log_2 \frac{1}{P(\V{x}^n)} -H  < \epsilon    \right \},
	\end{align*}
	where $P(\V{x}^n)=\mathbb{P}( \V{X}^n = \V{x}^n)$. 
	For all $\V{x}^n \in T^n_{1,\epsilon}$, it satisfies that
	\begin{align*}
	P(\V{x}^n) > 2^{-n (H+\epsilon)}. 
	\end{align*}
	And by \eqref{bit_left},
	\begin{align*}
	\mathbb{P} \left ( \V{X}^n \notin T^n_{1,\epsilon} \right  )  \leq  \exp \left ( -\frac{n \epsilon^2}{4 (\log_2 K)^2} \right ), 
	\end{align*}
	which implies $$\textnormal{for } n>\frac{4 (\log_2 K)^2 \log (1/\delta)}{\epsilon^2},$$ 
	$\mathbb{P} \left ( \V{X}^n \in T^n_{1,\epsilon} \right  ) \geq 1-\delta$. It further implies 
	$H_\delta(\mathcal{X}^n) \leq \log_2 |T^n_{1,\epsilon}| $ because $S_\delta^n$ is the smallest subset with probability greater than $1-\delta$. 
	
	Because for all $\V{x}^n \in T^n_{1,\epsilon}$, $ P(\V{x}^n)> 2^{-n (H+\epsilon)} $ and $\mathbb{P}(\V{X}^n \in T^n_{1,\epsilon} ) \leq 1$, we have 
	\begin{align*}
	|T^n_{1,\epsilon}| < 2^{n(H+\epsilon)}. 
	\end{align*}
	Therefore, 
	\begin{align*}
	H_\delta(\mathcal{X}^n) \leq \log_2 |T^n_{1,\epsilon}| <n(H+\epsilon), \textnormal{ for } n>\frac{4 (\log_2 K)^2 \log (1/\delta)}{\epsilon^2}. 
	\end{align*}
	
	We now prove the second part. 		Define $T^n_{2,\epsilon}$ as 
	\begin{align*}
	T^n_{2,\epsilon} = \left \{ \V{x}^n \in \mathcal{X}^n:   \frac{1}{n} \log_2 \frac{1}{P(\V{x}^n)} -H  > -\epsilon    \right \},
	\end{align*}
	Let $S'$ be any subset satisfying $|S'| \leq 2^{n(H-\epsilon)}$. Notice 
	\begin{align*}
	\mathbb{P}(\V{X}^n \in S') \leq \mathbb{P}(\V{X}^n \in S' \cap T^n_{2,\epsilon/2} )+\mathbb{P}(\V{X}^n \notin  {T^n_{2,\epsilon/2}} ). 
	\end{align*}
	The second term is bounded by $ \exp \left ( -\frac{n \epsilon^2}{16 (\log_2 K)^2} \right )$. We  bound the first term. 
	Because $|S' \cap T^n_{\epsilon/2}| \leq |S'| \leq 2^{n(H-\epsilon)} $ and $P(\V{x}^n) < 2^{-n (H-\epsilon/2)}$ for $\V{x}^n \in T^n_{2,\epsilon/2}$, 
	\begin{align*}
	\mathbb{P}(\V{X}^n \in S' \cap T^n_{2,\epsilon/2}) < 2^{n(H-\epsilon)} 2^{-n (H-\epsilon/2)} = 2^{-n \epsilon/2}. 
	\end{align*}
	It follow that
	\begin{align*}
	\mathbb{P}(\V{X}^n \in S') \leq 2^{-n \epsilon/2}+\exp \left ( -\frac{n \epsilon^2}{16 (\log_2 K)^2} \right ), 
	\end{align*}
	which implies $\mathbb{P} \left ( \V{X}^n \in S' \right  )< 1-\delta$,
	\begin{align*}
	\textnormal{for } n> \max \left \{ \frac{2\log_2(2/(1-\delta))}{\epsilon} ,\frac{16 (\log_2 K)^2 \log (2/(1-\delta))}{\epsilon^2} \right \}.
	\end{align*}
	Therefore,  $S_\delta^n$ must contain  more than $2^{n(H-\epsilon)}$ elements because it is required $\mathbb{P}( \V{X}^n \in S_\delta^n) \geq 1-\delta$, that is,
	\begin{align*}
	H_\delta(\mathcal{X}^n) > n(H-\epsilon).
	\end{align*}
\end{proof}
Intuitively speaking, the first part of Theorem \ref{thm:refine} tells us the number of bits does not need to exceed $n(H+\epsilon)$ for sufficiently large   $n$  even if the error $\delta$ is close to 0. Conversely, the second part tells us the number of bits cannot be smaller than $n(H-\epsilon)$ even if the error $\delta$ is close to 1. The technique used in the proof is called the method of \textit{typical sequences} \cite{cover2006elements}. A typical sequence is  $\V{x}^n$ that belongs to the \textit{typical set} $T^n_{1,\epsilon}\cap T^n_{2,\epsilon}$. The key ingredient of the proof of Theorem \ref{thm:refine} is the application of Corollary \ref{thm:bit}, which provides a much sharper bound for the probability of $\V{x}^n$ being outside of the typical set than the classical Chebyshev's inequality. The lower bound of $n$ therefore increases logarithmically as $\delta$ decreases in the first part and as $\delta$ increases in the second part, which are much slower than the rate  derived from Chebyshev's inequality. Moreover, the lower bound of $n$ increases on the order of $(\log_2 K)^2$ as $K$ grows, which is the slowest possible rate according to Theorem \ref{thm:counter_eg}. 

The method of typical sequences is a commonly-used proof technique for many important results in information theory \cite{cover2006elements,csiszar2011information}. The goal of this paper was exactly to bound the probability of such sequences. Specifically, we proved a rate-optimal exponential decay bound for the probability of a sequence not belonging to the typical set. Therefore, we expect that the new result and its generalizations can be used to sharpen the bounds in many information-theoretic results, for example, for variable-length codes and noisy channels. 

Next we use the method of typical sequences to prove a result on the error exponent of block codes in the high-dimensional setting.
The \textit{probability of error} of the code $(f,\varphi)$ is defined as \cite{csiszar2011information}:
\begin{align*}
e(f,\varphi)= \mathbb{P}( \varphi(f(\V{x}^n)) \neq \V{x}^n). 
\end{align*}
If $e(f,\varphi)$ drops as $e^{-n\alpha}$, $\alpha$ is called  \textit{error exponent}. Below we prove a result concerning $e(f,\varphi)$ and the error exponent, which are  uniform on $\V{p}$. 
\begin{proposition}\label{exponent}
	For $K\geq 5$ and sufficiently small positive $\epsilon>0$, there exists a block code $(f,\varphi)$ 
	\begin{align*}
	f: \mathcal{X}^n \rightarrow \{0,1\}^{\lceil n(H+\epsilon) \rceil}, \,\, \varphi: \{0,1\}^{\lceil n(H+\epsilon) \rceil} \rightarrow \mathcal{X}^n
	\end{align*}
	with the probability of error satisfying
	\begin{align*}
	e(f,\varphi) \leq \exp \left ( -\frac{n \epsilon^2}{4 (\log_2 K)^2} \right ).
	\end{align*}
\end{proposition}
\begin{proof}
	Recall 
	\begin{align*}
	T^n_{1,\epsilon} = \left \{ \V{x}^n \in \mathcal{X}^n:   \frac{1}{n} \log_2 \frac{1}{P(\V{x}^n)} -H  < \epsilon    \right \}. 
	\end{align*}
	It have been proved that   $|T^n_{1,\epsilon}| < 2^{n(H+\epsilon)}$. Therefore, one can construct a mapping $f$ that is one-to-one from  $T^n_{1,\epsilon}$ to $f(T^n_{1,\epsilon})$ and pick  an arbitrary string from $\{0,1\}^{\lceil n(H+\epsilon) \rceil} \backslash f(T^n_{1,\epsilon})$  for all $\V{x}^n \notin T^n_{1,\epsilon}$. Then  by \eqref{bit_left},
	\begin{align*}
	e(f,\varphi) \leq \mathbb{P} \left ( \V{x}^n \notin T^n_{1,\epsilon} \right  ) \leq  \exp \left ( -\frac{n \epsilon^2}{4 (\log_2 K)^2} \right ). 
	\end{align*}
\end{proof}
\textit{Remark:} The error exponent of the block code for fixed $K$ has been obtained   through a combinatorial argument instead of the method of typical sequences  (Theorem 2.15 in \cite{csiszar2011information}), which has the form 
\begin{align}
\inf_{Q: H(Q)\geq H+\epsilon} D(Q||P), \label{KL}
\end{align}
where $D(Q||P)$ is  the Kullback–Leibler divergence between distribution $Q$ and $P$. The probability of error satisfies
\begin{align}
e(f,\varphi) \leq \exp \left \{  -n \left [ \inf_{Q: H(Q)\geq H+\epsilon} D(Q||P)- \frac{\log (n+1)}{n} K \right ]  \right \}. \label{classical_exp}
\end{align}
This error exponent is proved to be optimal under the condition that  $K$ is fixed and $n$ goes to infinity \cite{csiszar2011information}. When $K$ is growing, the error exponent is still achievable if $K\log (n+1)/n \rightarrow 0$, which holds under the condition $(K^2\log K)/n \rightarrow 0$.  Therefore, despite being a uniform bound for the probability of error,  Theorem \ref{thm:old} (\cite{zhao2020note}, Corollary 1) does not provide a better error exponent. 
However, when $K\log (n+1)/n= \Omega(1)$,  the error exponent \eqref{KL} might not be achieved and  \eqref{classical_exp} may even blow up. By contrast, the new result  still gives an exponential decay bound as long as $(\log K)^2/n \rightarrow 0$.

\section{Conclusion}
We proved a uniform concentration bound for the tail probability of log-likelihoods of discrete random variables. The key steps in the proof are to bound the variance of the log-likelihood $\textnormal{Var}( L(\V{z}_i))$ (Lemma \ref{control_var}) and the MGF $M_Y(\lambda,\V{p}_i)$ (Lemma \ref{thm:opt2}). We proved the two bounds by viewing them  as optimization problems and applying the primal-dual method. Essentially, we proved the duality gaps are zero under certain conditions by techniques in real analysis. Furthermore, we gave examples of the applications of the new result in information theory. 

One direction we are exploring is to generalize the result to discrete variables with countably infinite number of values. It is known that not every discrete variable with a countably infinite number of values has a finite entropy \cite{baccetti2013infinite}. Even within the class of distributions that have finite entropies, a uniform concentration bound over the class does not exist. The counterexample in Theorem \ref{thm:counter_eg} implies that the bound becomes trivial if the number of non-zero probabilities goes to infinity. Therefore, to figure out proper constraints to be put on the class is an intriguing question.  Another direction is to generalize the results to non-independent variables, such as martingales and weakly dependent variables. Moreover, researchers may be interested in applying the new result to other problems in information theory, for example, to variable-length codes and noisy channels, particularly in the high-dimensional setting. 

\appendix

\section*{Appendix}
\begin{lemma}\label{thm:tedious}
	Let  $\nu=-(\lambda+1)(1/K)^\lambda+\lambda(-\log K +1)$, and $f(p)= p^{\lambda+1}-\lambda  p \log p+\nu p$, $p\geq 0, \lambda>-1$.  Then for $K\geq 5$, $\lambda>-1$ and $\lambda \geq \frac{2-2/K-\log K}{(1-1/K) \log K }$, $f(1/K)\geq f(1)$ and $f(1/K) \geq f(0)$.
\end{lemma}
\begin{proof}

	It is easy to verify that
	\begin{align*}
	f(0) & =  0, \\
	f(1) & =  1+\nu =1-(\lambda+1)(1/K)^\lambda+\lambda(-\log K +1), \\
	f(1/K) &=  (1/K)^{\lambda+1}-\lambda (1/K) \log(1/K)+\nu (1/K)=(1/K) \lambda \left (1-(1/K)^\lambda \right )\geq 0 =f(0).
	\end{align*}
	Therefore, we only need to prove $f(1/K) \geq f(1)$. Let 
	\begin{align*}
	g(\lambda) = &f(1/K)-f(1) \\
	= & (\lambda-\lambda/K+1) \exp(-\lambda \log K)+\lambda(1/K+\log K-1)-1.
	\end{align*}
	Its first and second derivatives are
	\begin{align*}
	g'(\lambda) = & (1-1/K)\exp(-\lambda \log K)+(\lambda-\lambda/K+1) \exp(-\lambda \log K) (-\log K)+1/K+\log K-1, \\
	g''(\lambda)= & (1-1/K)\exp(-\lambda \log K) (-\log K)+(1-1/K-(\lambda-\lambda/K+1)(\log K)) \exp(-\lambda \log K)(-\log K).
	\end{align*}
	Note that $\frac{2-2/K-\log K}{(1-1/K) \log K } <0$ for $K\geq 5$, i.e., $0\in \left [\frac{2-2/K-\log K}{(1-1/K) \log K },\infty \right )$.	Moreover, it is easy to verify that $g(0)=0$ and $g'(0)=0$. 
	Therefore, to prove $g(\lambda)\geq 0$ for $\lambda \geq \frac{2-2/K-\log K}{(1-1/K) \log K }$, i.e., $g(0)$ is the minimum in that range, we only need to show  $g$ is convex, i.e., $g''(\lambda) \geq 0$ for $\lambda \geq \frac{2-2/K-\log K}{(1-1/K) \log K }$. In fact,
	\begin{align*}
	& g''(\lambda) \geq 0 \\
	\Leftrightarrow & 1-1/K+1-1/K-\log K-\lambda(1-1/K) \log K \leq 0 \\
	\Leftrightarrow & \lambda \geq \frac{2-2/K-\log K}{(1-1/K) \log K }. 
	\end{align*}		
\end{proof}

\section*{Acknowledgements}
This research was supported by the National Science Foundation grant DMS-1840203.

 \newcommand{\noop}[1]{}

\end{document}